\documentclass[12pt]{amsart}
\usepackage{graphicx}
\usepackage{xcolor}
\usepackage{url}
\usepackage{amsmath, amssymb, mathtools}

\usepackage{graphpap,color,paralist,pstricks}
\usepackage[pdftex,colorlinks,backref=page,citecolor=blue]{hyperref}

\usepackage{fullpage}
\usepackage{amscd,amsthm}
\usepackage{comment}
\usepackage{color}
\usepackage{mathtools}
\usepackage{geometry}

\usepackage[utf8]{inputenc}

\usepackage{amssymb}

\def\RR{\mathbb{R}}

\def\M{\mathcal{M}}
\def\B{\mathcal{B}}
\def\U{\mathcal{U}}
\def\F{\mathcal{F}}

\def\S{\mathfrak{S}}

\theoremstyle{plain}

\newtheorem{theorem}{Theorem}[section]
\newtheorem{proposition}[theorem]{Proposition}

\newtheorem{lemma}[theorem]{Lemma}
\theoremstyle{definition}
\newtheorem{remark}[theorem]{Remark}
\newtheorem{example}[theorem]{Example}
\newtheorem{definition}[theorem]{Definition}

\newtheorem{problem}[theorem]{Problem}

\title{On subdivisions of the permutahedron and flags of lattice path matroids}
\author[C. Benedetti]{Carolina Benedetti-Velásquez}
\address[C. Benedetti]{Department
of Mathematics, Universidad de los Andes}
\email{c.benedetti@uniandes.edu.co}
\begin{document}

\begin{abstract}
In this manuscript we study the subdivisions of the permutahedron $\Pi_n$ into two subpolytopes corresponding to flags of positroids, which are in particular flags of lattice path matroids (LPFMs). A subpolytope $P_{[u,v]}$ of $\Pi_n$ is a Bruhat Interval Polytope (BIP) if $P_{[u,v]}$ is the convex hull of all the permutations (viewed as points in $\RR^n$) in the interval $[u,v]$ in the Bruhat order of $\S_n$. We show that the coarsest subdivisions we obtain into LPFMs are the only subdivisions of $\Pi_n$ via hyperplane splits, into subpolytopes corresponding to BIPs. 
More specifically, we describe the hyperplanes whose intersection with $\Pi_n$ give rise to BIPs. Hence, these subdivisions are polytopes coming from points in the complete nonnegative flag variety. 
\end{abstract}

\maketitle

\section{Introduction}\label{sec:intro}

Let $V$ be a $k$-dimensional vector space in $\RR^n$ and let $\{v_1,\dots,v_k\}$ be a basis of it. That is, $V$ is a point in the real Grassmannian $Gr_{k,n}$. Set $A$ to be the matrix whose row $i$ is $v_i$, for each $i$. Then the collection $\B$ of $k$-subsets of the $n$ columns of $A$ that form a basis for its column space give rise to a representable (over $\RR$) matroid $M=([n],\B)$, where $[n]=\{1,\dots,n\}$. In this case we say that $A$ represents the matroid $M$, although notice that such $A$ is not  unique. The matroid $M$ over $[n]$ is said to have rank $k$. 
If $M=([n],\B)$ is any matroid, its \emph{matroid (base) polytope} is the convex hull of the indicator vectors in $\RR^n$ of each $B\in\B$. It is also the moment map image of the closure of the torus orbit of $A$ in $Gr_{k,n}$, see \cite{GELFAND1987301}.

A particular class of representable matroids (over $\RR$) known as \emph{positroids} led to a stratification of the totally nonnegative Grassmannian $Gr^{\geq 0}_{k,n}$. A positroid $P$ of rank $k$ over $[n]$ can be thought of as a matroid for which there exists a full rank $k\times n$ matrix $A$ whose maximal minors are non-negative, such that $A$ represents $P$ \cite{positivity,BLUM2001937}. 
 A key example of positroids is provided by the family of \emph{lattice path matroids} LPMs, introduced by Bonin, de Mier, and Noy in \cite{BdMN2003}.
 For a fixed $n$, the data of a lattice path matroid $M = M[U,L]$ is given by an `upper' path $U$ and a `lower' path $L$, both lattice paths from $(0,0)$ to $(n-k,k)$ for some $k\leq n$. The bases of $M$ are given by all the lattice paths from $(0,0)$ to $(n-k,k)$ that lie in between $U$ and $L$. See Figure \ref{fig:xmpl} for an example.
 In particular, the \emph{uniform matroid} $U_{k,n}$ corresponds to the LPM whose diagram is the rectangle from $(0,0)$ to ${n-k,k}$.
From a geometric perspective, an LPM corresponds to a generic point in a cell arising from the Richardson cell decomposition of the Grassmannian $Gr_{k,n}$.

A prequel of this manuscript is given by Benedetti and Knauer in \cite{BenKna}. There, the authors study \emph{flags} or \emph{quotients} of LPMs. Given two matroids $M,N$ on the ground set $[n]$ one say that $M$ is a quotient of $N$, or that $\F:(M,N)$ is a flag matroid, if every circuit of $N$ is union of circuits of $M$. We denote this as $M\leq_q N$, and we refer to $M$ and $N$ as the \emph{constituents of the flag $\F$}. As expected, there are several equivalent ways to state that $M\leq_q N$. We provide some of these equivalences in Definition \ref{def:matroid_quotient}.

A very natural example of a flag matroid is given by sequences of vector spaces. In particular, consider the \emph{(real) full flag variety $\F l_n$} which consists of all sequences $F:V_{0}\subset V_1\subset\cdots\subset V_{n}$ of $\mathbb R$-vector spaces where $\dim V_i=i$ for all $i$. Now let $M_i$ be the matroid represented by the vector space $V_i$. Then the collection of matroids $(M_0,M_1,\cdots,M_n)$ give rise to a full flag matroid where each $M_i$ is a \emph{quotient} of $M_{j}$ for $1\leq i\leq j\leq n$. 
Hence, points in $\F l_n$ give rise to full flag matroids, although it is not the case that every full flag matroid comes from a point in $\F l_n$.

Now, the \emph{nonnegative (full) flag variety} $\F l_n^{\geq 0}$ consists of flags $F:V_{0}\subset V_1\subset\cdots\subset V_{n}\in\F l_n$ for which there exists a matrix $A_F$ that realizes $F$ with the additional property that for each $i$, the submatrix given by the top $i$ rows of $A_F$ has nonnegative maximal minors. From this it follows that the flag matroid $\F:(M_0,M_1,\dots,M_n)$ arising from $F$ has the property that each constituent $M_i$ is a positroid. If a flag of matroids $(M_0,M_1,\dots,M_n)$ has the property that each $M_i$ is a positroid, we refer to it as a \emph{positroid (full) flag}. When each constituent $M_i$ is an LPFM we say that the flag $F$ is an LPFM. 
We emphasize that it is \emph{not} the case that every positroid flag $M_0\leq_qM_1\leq_q\cdots\leq_q M_n$ corresponds to a point in $\F l_n^{\geq 0}$, see \cite[Example 7]{BenKna}.

Given a flag of matroids $\mathcal M:M_0\leq_qM_1\leq_q\cdots\leq_q M_n$, the Minkowski sum of the matroid polytopes $P_{M_i}$ is the \emph{flag matroid polytope $P_{\mathcal M}$}. Such a polytope  $P_{\mathcal M}$ has vertices given by  certain permutations of the point $(1,2,\dots,n)$. A motivation to write this manuscript was to understand the relation of flag matroid polytopes whose constituents are LPMs and the permutahedron.

The permutahedron $\Pi_n$ is a polytope in $\RR^n$ that can be obtained in various ways: (i) as the convex hull of all the permutations $z=z_1z_2\cdots z_n\in\S_n$ viewed as points $(z_1,z_2,\ldots, z_n)\in\RR^n$; (ii) as the moment map image of the complete nonnegative flag variety $\F l_n^{\geq 0}$ (\cite{KodaWill_flag,BEW24}); (iii) as the Minkowski sum $\Delta_{1,n}+\Delta_{2,n}+\cdots+\Delta_{n,n}$ (\cite{KodaWill_flag}) where $\Delta_{k,n}$ is the $(k,n)$-hypersimplex obtained as the matroid polytope of $U_{k,n}$.

Perspective (i) tells us that the permutahedron is an example of a \emph{Bruhat interval polytope (BIP)} as defined in \cite{KodaWill_flag}. A BIP is a polytope $P_{[u,v]}$ whose vertices are all the vertices in a Bruhat order interval $[u,v]$. Thus, $\Pi_n$ is the BIP indexed by the whole interval $[e,w]$ from the identity to the longest permutation of $\S_n$. Perspective (ii) tells us that $\Pi_n$ is the flag matroid polytope of a flag matroid whose constituents $(M_1,\dots,M_n)$ are positroids and, moreover, for which there is a matrix $A\in \F l_n^{\geq 0}$ whose top $i$-rows realize $M_i$ as a positroid (\cite[Definition 2.6]{BEW24}). Perspective (iii) tells us that $\Pi_n$ is the flag matroid polytope of the flag matroid whose constituents are the uniform matroids of ranks $1,2,\dots,n$ over $[n]$.

Going back to \cite{BenKna}, Benedetti and Knauer considered quotients of LPMs that are themselves LPMs. They provided a combinatorial condition to identify when an LPM is a quotient of another LPM as illustrated in Figure \ref{fig:goodpair} below.
In particular, \cite{BenKna} proves that a full flag of LPMs $L_0\leq_q L_1\leq_q\cdots\leq_q L_n$ comes from a point in $\F l_n^{\geq 0}$. This is done by showing that the corresponding flag matroid polytope is a BIP.

The main contributions can be summarized as follows: in Theorem \ref{thm:coarsest_v1} we provide a collection of hyperplanes parallel to some facets of $\Pi_n$. Each such hyperplane subdivides $\Pi_n$ into two BIPs, each of them being the flag polytope of a LPFM. On the other hand, in Theorem \ref{thm: all_split_hyperplanes} we show that the hyperplanes of Theorem \ref{thm:coarsest_v1} are the only ones subdividing $\Pi_n$ into BIPs.

One of the main reasons to care about subdivisions of $\Pi_n$ into BIPs comes from tropical geometry. In \cite{JLLO,BEW24} the authors show that such subdivisions can be parametrized by points in the nonnegative tropical flag variety $Tr^{>0}\F l_n$. Hence, we are providing a large family of points in $Tr^{>0}\F l_n$ using LPFMs. In fact, the spark that ignited this paper comes from Table \cite[Table 2]{BEW24} where the authors computed the coarsest (nontrivial) subdivisions of $\Pi_n$ into BIPs and this lead us to wonder about the ones coming from LPFMs. 

This manuscript is organized as follows. In Section \ref{sec:prelim} we provide the background needed to dive into the paper. Passing from matroids and their polytopes to flags of matroids and BIPs. In Section \ref{sec:main_results} we state and prove the main results as explained briefly here. Finally, in Section \ref{sec:problems} we state several problems that we are interested on pursuing further.

\section{Acknowledgements}
We thank Chris Eur for kindly providing his github package on computations of $Tr^{>0}\F l_5$. Also we are very grateful with Johannes Rau for kindly explaining the tropical context relevant to us, and for posing interesting questions. Finally we thank Kolja Knauer for many fruitful discussions in this sequel, let alone all the work in the prequel.

\section{Preliminaries}\label{sec:prelim}
Denote by $[n]$ the set $\{1,2,\dots,n\}$ and by $\displaystyle{{[n]\choose k}}$ the subsets of size $k$ of $[n]$, or $k$-subsets of $[n]$.

\subsection{Matroids and positroids}

We first recall some relevant notions from matroid theory. We assume the reader is familiar with basic properties properties of matroids, we refer to \cite{Oxley} for any undefined terms. 

\begin{definition}\label{def:matroid_quotient}
 A \textit{matroid} $M$ on $[n]$ is a nonempty collection $\mathcal{B}(M)$ of subsets of $[n]$ satisfying the exchange axiom: for any $B_1,B_2\in\mathcal{B}(M)$ and $x\in B_1\backslash B_2$, there exists $y\in B_2\backslash B_1$ such that $(B_1\backslash \{x\})\cup \{y\}\in \mathcal{B}(M)$. Elements of $\mathcal{B}$ are called \textit{bases} and we write $M=([n],\B)$.
    \end{definition}  

The \emph{ (real) Grassmannian} $Gr_{k,n}$ is the set of $k$-dimensional vector spaces $V$ in $\RR^n$. Such $V$ can be thought of as a full rank $k\times n$ matrix $A$ whose rows are a basis for $V$. In this case we say that $A$ is a point in $Gr_{k,n}$.
Given such $A$ we denote by $p_A:=(p_I)_I$ its \emph{Plücker vector}, where $I\in\displaystyle{{[n]\choose k}}$ and $p_I$ is the \emph{Plücker coordinate} obtained as the maximal minor of $A$ computed using columns $I=\{i_1,\dots,i_k\}$ of $A$. 

A matroid $M=([n],\B)$ is \emph{representable} (over $\RR$) if there exists a point $A\in Gr_{k,n}$ such that $p_I\neq 0$ if and only if $I\in\B$. In this case we say that $k$ is the \emph{rank} of $M$, and write $r_M=k$.

The \emph{non-negative Grassmannian} $Gr_{k,n}^{\geq 0}$ is the set of elements $V\in Gr_{k,n}$ for which there is a matrix representation $A$ whose Plücker coordinates are all non-negative. In this case we say that $A$ is a point in $Gr_{k,n}^{\geq 0}$.

\begin{definition}\label{def:positroid}
     A \emph{positroid} is a matroid $M=([n],\B)$ for which there is a point $A\in Gr_{k,n}^{\geq 0}$ that represents $M$.
\end{definition}

As an example, the \emph{uniform matroid} $U_{k,n}$ of rank $k$ over $[n]$ is the matroid whose bases are all the $k$-subsets of $[n]$. It turns out that uniform matroids are a particular example of the family of \emph{lattice path matroids} (LPMs) which we define next.

 For $k\leq n$ let $U=\{u_1<\cdots<u_k\}$ and $L=\{l_1<\cdots<l_k \}$ be elements of $\displaystyle{{[n]\choose k}}$. The \emph{Gale order} $\leq_G$ on $\displaystyle{{[n]\choose k}}$ is such that $U\leq_G L$ if $u_i\leq l_i$ for all $i$. 

 \begin{definition}
  For $k\leq n$ let $U,L\in\displaystyle{{[n]\choose k}}$ such that $U\leq_G L$. The \emph{lattice path matroid} $M=M[U,L]$ is the matroid on $[n]$ whose collection of bases $\B$ consists of elements $B\in\displaystyle{{[n]\choose k}}$ such that $U\leq_G B\leq_G L$.
 \end{definition}

Lattice path matroids (LPMs) were defined as such in \cite{Bon-03}. It has been proved in different ways that they are indeed positroids (see \cite[Lemma 23]{AyushAshraf23}). Graphically, a lattice path matroid $M=M[U,L]$ can be visualized by drawing in the first quadrant the lattice paths $U$ and $L$ from $(0,0)$ to $(n-k,k)$ whose north steps are the ones taken at times $u_1,\dots,u_k$ and $l_1,\dots,l_k$, respectively. See Figure \ref{fig:xmpl} for an example of an LPM (hence a positroid) on $[8]$ of rank 4.

\begin{figure}[htp]
    \centering
    \includegraphics[width=.3\textwidth]{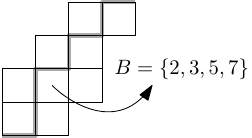}
    \caption{A basis in the diagram of the LPM $M[1246,3568]$.}\label{fig:xmpl}
\end{figure}

\begin{remark}
    An LPM on $[n]$ of rank $k$ can also be thought as a generic point in the corresponding Richardson cell of $Gr_{k,n}$ (see~\cite{KodaWill_gras}).
\end{remark}

Just like points in $Gr_{k,n}$ ($Gr_{k,n}^{\geq 0}$) correspond to representable matroids (positroids), consider now a flag of vector spaces $\F:V_1\subset\cdots\subset V_{s-1}\subset V_s$ in $\RR^n$ of dimensions $d_1<\cdots<d_{s-1}<d_s$. Such a flag $\F$ can be thought of as a full rank $d_s\times n$ matrix $A$ whose first $d_i$ rows are a basis for $V_i$, for all $i$. Hence, such a matrix $A$ \emph{represents} a sequence of matroids $(M_1,\dots,M_{s-1},M_s)$ on $[n]$ where $M_i$ is the matroid represented by the submatrix $A_i$ formed by the first $d_i$-rows of $A$. The flag $\F$ is a \emph{full flag} if $s=n$. The \emph{partial flag variety $\F l_n(d_1,\dots,d_s)$} consists of flags of vector spaces $F:V_1\subset\cdots\subset V_{s-1}\subset V_s$ as above. When represented via a matrix $A$ as we just described we say that $A$ is a point in $\F l_n(d_1,\dots,d_s)$. The full flag variety will be simply denoted $\F l_n$. 
The correspondence between points $A\in \F l_n(d_1,\dots,d_s)$ and sequences  of matroids $(M_1,\dots,M_s)$ that we just discussed is captured and generalized in the following central definition.

\begin{definition}\label{def:quotient}\cite[Prop. 7.4.7]{White}
    Let $M$ and $N$ be matroids on $[n]$. The matroid $M$ is a \emph{quotient} of the matroid $N$, denoted $M\leq_q N$, if any of the following equivalent conditions is satisfied:
    \begin{enumerate}
        \item Every circuit of $N$ is union of circuits of $M$.
        \item Every flat of $M$ is a flat of $N$.
        \item For every $B\in\B(N)$ and every $p\notin B$ there is $B'\in\B(M)$ such that $B'\subseteq B$ and if $B'\setminus\{p\}\cup\{q\}\in\mathcal B(M)$ then $B\setminus\{p\}\cup\{q\}\in\mathcal B(N)$ for all $q\in B'$.
    \end{enumerate}    
    A sequence of different matroids $\M:(M_1,\dots,M_s)$ on $[n]$ such that $M_i\leq_{q}M_{i+1}$ for $i\in[s]$ is a \emph{flag matroid} and the $M_i$ are its \emph{constituents}. If each $M_i$ is a positroid we say that $\M$ is a \emph{positroid flag}. 
    When $s=n$ the flag is a \emph{full flag matroid}.
\end{definition}

\begin{example}\label{ex:flag_mat}
\begin{enumerate}
    \item As suggested in the previous paragraph, let $N$ be the rank $4$ matroid on $[7]$ represented by
    $$A=\begin{pmatrix}
        1&0&1&0&1&1&1\\0&1&1&0&2&2&1\\0&0&0&1&1&2&1\\0&0&0&0&0&0&1
    \end{pmatrix}.$$ Then the matroid $M$ represented by the submatrix $\begin{pmatrix}
        1&0&1&0&1&1&1\\0&1&1&0&2&2&1
    \end{pmatrix}$ is a quotient of $N$, and thus $A$ can be thought of as a point in the partial flag variety $\F l_n(2,4)$. We leave the reader to verify either of the three conditions in Definition~\ref{def:quotient}.
    \item The \emph{(full) uniform flag on $[n]$} is the full flag matroid $\U$ whose $i$-th constituent is the uniform matroid $U_{i,n}$, for every $i\in [n]$. 
\end{enumerate}

\end{example}

\begin{remark}\label{rem:not_rep_flag}
   A flag matroid $(M_1,\dots,M_s)$ on $[n]$ for which there exists a point $A\in\F l_n(d_1,\dots,d_s)$ such that $A_i$ (the first $d_i$ rows of $A$) represent $M_i$ is called a \emph{representable flag matroid}. Example \ref{ex:flag_mat} (1) is one such flag.
    Given a flag matroid $\M:(M_1,\dots,M_s)$ on $[n]$ such that each constituent $M_i$ is representable, it does not necessarily follow that $\M$ can be represented by a point $A\in\F l_n(d_1,\dots,d_s)$, where $d_i=r(M_i)$. In other words, a flag matroid build with representable matroids, does not guarantee that the flag matroid is representable (see~\cite[Section 1.7.5]{Borovik2003} or~\cite[Example 6.9]{am2018flag}).
\end{remark}

In this paper we are concerned with flag matroids whose constituents are positroids. More specifically, we will explore the relation between full flag matroids on $[n]$ whose constituents are LPMs, and subdivisions of the \emph{permutahedron} $\Pi_n$, which we will define shortly.

\begin{definition}\cite[Definition 2.2]{BEW24}\label{def:nonnegflag}
    The \emph{nonnegative (full) flag variety} $\F l_n^{\geq 0}$ consists of the points $A\in\F l_n$ such that the submatrix $A_i$ build by taking the first $i$ rows of $A$, is a point in $Gr_{i,n}^{\geq 0}$, for all $i\in[n]$. 
\end{definition}

 In view of Remark~\ref{rem:not_rep_flag} one may wonder: given a positroid flag $\M:(M_1,\dots,M_s)$ on $[n]$ of ranks $d_1<\cdots<d_s$, is there a matrix $A\in\F l_n^{\geq 0}$ that represents $\M$? That is, such that $A_i\in Gr_{d_i,n}^{\geq 0}$ represents $M_i$? Here, $A_i$ is the submatrix of $A$ given by taking its first $d_i$ rows as before. The answer in general is negative (see~\cite[Example 7]{BenKna}), although is affirmative if the $M_i's$ are LPMs.

In~\cite{BenKna} the characterization of quotients of LPMs was made using (3) of Definition~\ref{def:quotient}. We state this characterization here as we will be using it.

Given an LPM $M=M[U,L]$ on $[n]$ where $U=\{u_1<\cdots<u_k\}$ and $L=\{l_1<\cdots<l_k\}$, we say that a pair $(u_j,l_i)$ is a \emph{good pair of $M$} if and only if $\max\{0,u_j-\ell_i\}\leq j-i$.
Graphically, being a good pair can be visualized as follows. Let $(a,b)$ be the coordinates of the northern vertex of the step $u_j$. If $l_i$ is a step on $L$ within the region that lies in between the halfspaces $x\geq a$ and $y\leq b$, then $(u_j,l_i)$ is a good pair. That is, after fixing $u_j$ we obtain the set of steps in $L$ that form a good pair with $u_j$, and every good pair arises this way. In Figure~\ref{fig:goodpair} we illustrate in the shaded region all the possible good pairs of the form $(4,-)$ of the LPM $M[1247,3568]$.
Now, in Theorem~\cite[Prop. 12]{BenKna} the authors prove that if $M=M[U,L]$ as above then every good pair $(u,l)$ of $M$ gives rise to a matroid $M'=M[U-u,L-l]$ such that $M'\leq_q M$. In this case, the quotient $M'$ is an \emph{elementary} quotient since the ranks of the matroids differ by 1. If $M'\leq_q M$ is not elementary, Theorem 19 of ~\cite{BenKna} proves that there is a flag of LPMs $M´=M_0\leq_qM_1\leq_q\cdots\leq_qM_s=M$ such that the intermediate quotients are elementary. If $\M:M_1\leq_q\cdots\leq_qM_n$ is a full flag of LPMs on $[n]$, we will say that $\M$ is an LPFM, for short.

In Figure~\ref{fig:goodpair} we illustrate an elementary quotient of $M[1247,3568]$ after removal of the good pair $(4,5)$, as well as a quotient of it that is not elementary.

\begin{figure}[htp]
    \centering
    \includegraphics[width=.8\textwidth]{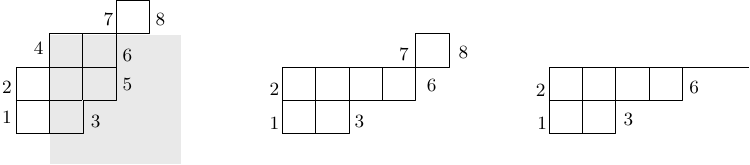}
    \caption{Left: $M=M[1247,3568]$. Center: elementary quotient of $M$. Right: a (not elementary) quotient of $M$.}\label{fig:goodpair}
\end{figure}

\begin{remark}
  Definition~\ref{def:matroid_quotient} is a particular case of \emph{strong maps} between matroids (see~\cite{White}). Our choice of notation $M\leq_q N$ whenever $M$ is a quotient of $N$ is inspired by the fact that matroids on a fixed ground set can be endowed with a partial order structure $\leq$ such that $M\leq N$ if and only if $M\leq_q N$. In~\cite{BCT} the authors pose the question of whether the order $\leq$ restricted to positroids gives rise to a ranked poset (as it it for matroids). In~\cite{BenKna} the authors answer this affirmatively in the case of LPMs. As far as we are concerned this question is still not answered in the case of positroids.
\end{remark}

We now dive into the polytopal analysis of flag matroids, in order to better understand the behavior of flags of LPMs. We will only consider full flags of positroids. More specifically, we will utilize mainly full flags of LPFMs to which we will refer to as Lattice Path Flag Matroids, or LPFMs, for short.

\subsection{The polytope of a flag matroid}
We will assume the reader is familiar with terminology of polytopes and the symmetric group. We encourage the reader to check \cite{ziegler1995lectures} and \cite{PostPerm,PostReinerWill}, for a deeper understanding of this section.

For $n\geq 1$ fixed denote by $\S_n$ the symmetric group on $n$ elements. The simple transpositions of $\S_n$, namely $s_i=(i,i+1)$ for $i=1,\dots,n-1$, generate $\S_n$ as a group. The \emph{length} of a permutation $u\in \S_n$ is the minimum number $k$ of generators needed to obtain an expression for $u$. We denote this as $\ell(u)=k$. The \emph{(strong) Bruhat order} is the partial order on $\S_n$ given as the transitive closure of the covering relations
$$u\lessdot v\Leftrightarrow \ell(v)=\ell(u)+1 \text{ and }(i,j)u=v  \text{ for some }1\leq i< j\leq n.$$

Given $u\in\S_n$ we will write $u$ in one-line notation as $u=a_1\cdots a_n$ if $u(i)=a_i$.

Now let $\Pi_n$ denote the $n$\emph{-permutahedron}. That is, 
$\Pi_n$ is the polytope in $\mathbb R^n$ whose $V$-description is given by
$$ \Pi_n=\text{conv}\{(a_1,\dots,a_n)\,|\,a_1\cdots a_n\in S_n\}.
$$
Alternatively, the $H$-description of $\Pi_n$ may be given as 
\begin{equation}\label{eq:Pi_n}
    \Pi_n=\left\{(x_1,\dots,x_n)\,|\,x_1+\cdots+x_n=1+2+\cdots+n,\,\displaystyle\sum_{S\subset[n]}x_S\geq {{|S|+1}\choose 2} \right\}
\end{equation}
where $x_S:=x_{i_1}+\cdots+x_{i_m}$ if $S=\{i_1,\dots,i_m\}$.

Equation~\ref{eq:Pi_n} describes in fact the facets of $\Pi_n$ as being the faces of $\Pi_n$ obtained by intersecting $\Pi_n$ with hyperplanes of the form $\displaystyle\sum_Sx_S={{|S|+1}\choose 2}$, where $S$ is a non-empty proper subset of $[n]$. Therefore, $\Pi_n$ has $2^n-2$ facets and each facet can be labelled by one such subset of $[n]$.

Let $M=([n],\mathcal B)$ be a matroid. The \emph{matroid (base) polytope of $M$} is the polytope $P_M$ in $\RR^n$ given by
\begin{equation*}
    P_M:=\text{conv}\{e_B:B\in\B\}
    \end{equation*}
    where $e_B$ is the indicator vector of $B$. In fact, $\{e_B:B\in\B\}$ is the collection of vertices of $P_M$. 

When $M=U_{k,n}$ is the uniform matroid of rank $k$ on $[n]$, the matroid polytope $P_M$ receives the name of the $(k,n)$-hypersimplex and its denoted $\Delta_{k,n}$. Notice then that if $N$ is another matroid on $[n]$ of rank $k$ then $P_N\subseteq\Delta_{k,n}$ since every vertex of $P_N$ is a vertex of $P_M$.

\begin{definition}\cite{Borovik2003}\label{def:flag_mat_pol}
    Let $\M\!:\!(M_1,\cdots,M_k)$ be a flag matroid on the ground set $[n]$. The \emph{flag matroid polytope $P_{\M}$ of $\M$} is the polytope given as the Minkowski sum 
    \[
    P_{\M}:=P_{M_1}+\cdots+P_{M_k}.
    \]
\end{definition}

An alternative way to describe the polytope $P_{\M}$ is via its vertices. For each $i=1,\dots,k$ let $\B_i$ be the collection of bases of the matroid $M_i$. Consider the collection of \emph{flags of bases} $\B\M=\{B_1\subset\cdots\subset B_k :B_i\in\B_i, i\in[k] \}$. The vertices of the polytope $P_{\M}$ are given by the set 
\[
\{e_F:F\in\B\M\}
\]
where $e_F=e_{B_1}+\cdots+e_{B_k}$ is the indicator vector of the flag of bases $F:B_1\subset\cdots\subset B_k$. This implies that the polytope $P_{\M}$ lies in the hyperplane $x_1+\cdots+x_n=r_1+\cdots+r_n$ where $r_i=r(M_i)$ is the rank of the matroid $M_i$.

Let us illustrate flag matroid polytopes with the following examples.
\begin{example}\label{ex:flag_polytope}
The following two examples are illustrated in Figure \ref{fig:flag_example_polytope}, on the left side and right side, respectively.

\noindent$(i)$ The permutahedron: consider the (full) flag $\U_3:(U_{1,3},U_{2,3},U_{3,3})$ of uniform matroids on $[3]$. The flag matroid polytope of $\U_3$ is the permutahedron $\Pi_3$. Indeed, the six flags of bases in $\B\U_3$ are
        \begin{align*}
F_1:1\subset 12\subset 123\quad F_2:1\subset 13\subset 123\quad       F_3: 2\subset 12\subset 123\\ F_4:2\subset 23\subset 123\quad
F_5:3\subset 13\subset 123\quad F_6: 3\subset 23\subset 123
        \end{align*}
 whose corresponding vertices are
  \begin{align*}
e_{F_1}=(3,2,1)\quad e_{F_2}=(3,1,2)\quad e_{F_3}=(2,3,1)\\
e_{F_4}=(1,3,2)\quad e_{F_5}=(2,1,3)\quad e_{F_6}=(1,2,3)
        \end{align*}
It is an exercise for the reader to check that $P_{\U_n}=\Pi_n$, for any $n\geq 1$.

\noindent$(ii)$ Let $\M$ be the LPFM whose constituents are $M[2,4]$, $M[12,24]$ and $M[124,234]$ as shown in Figure \ref{fig:flag_example_polytope}.
Two vertices of $P_{\M}$ come from the flags $2\subset 12\subset 124$ of the $U$-bases of each constituent and $4\subset 24\subset 234$ of the $L$-bases of each constituent. These flags of bases give rise to the vertices $v=(2,3,0,1)$ and $u=(0,2,1,3)$, respectively. We can assume that the flag matroid $\M$ has as last constituent the uniform matroid $U_{4,4}$ which has the effect of translating the polytope $P_{\M}$ by the vector  of ones $(1,1,1,1)$ giving us $v=(3,4,1,2)$ and $u=(1,3,2,4)$. In the Figure, color red highlights the vertex $u$ corresponding to the $L$-bases, whereas blue highlights $v$ corresponding to the $U$-bases. We encourage the reader to verify that all the vertices of $P_{\M}$ are the permutations in the interval $[u,v]$. This is not an accident as we will unveil shortly.

\end{example}

Let $F:B_1\subset B_2\subset\cdots\subset B_n=[n]$ be a full flag of subsets of $[n]$, where $B_i\setminus B_{i-1}=\pi_i$. Then $\pi=\pi_1\pi_2\cdots\pi_n$ is a permutation of $[n]$. Let $\tau_F:=\tau_1\tau_2\cdots\tau_n\in S_n$ be such that $\tau:=\pi_c^{-1}$ where $\pi_c(i)=\pi_{n-i+1}$, for each $i\in [n]$. We refer to $\tau$ as the \emph{Bruhat permutation of the flag $F$}. For example, if $F:1\subset 13\subset 135\subset 1345\subset 12345$ then $\pi=13542, \pi_c=24531$ and the Bruhat permutation is $\tau_{F}=51423$.

\begin{figure}[htp]
    \centering
    \includegraphics[width=.8\textwidth]{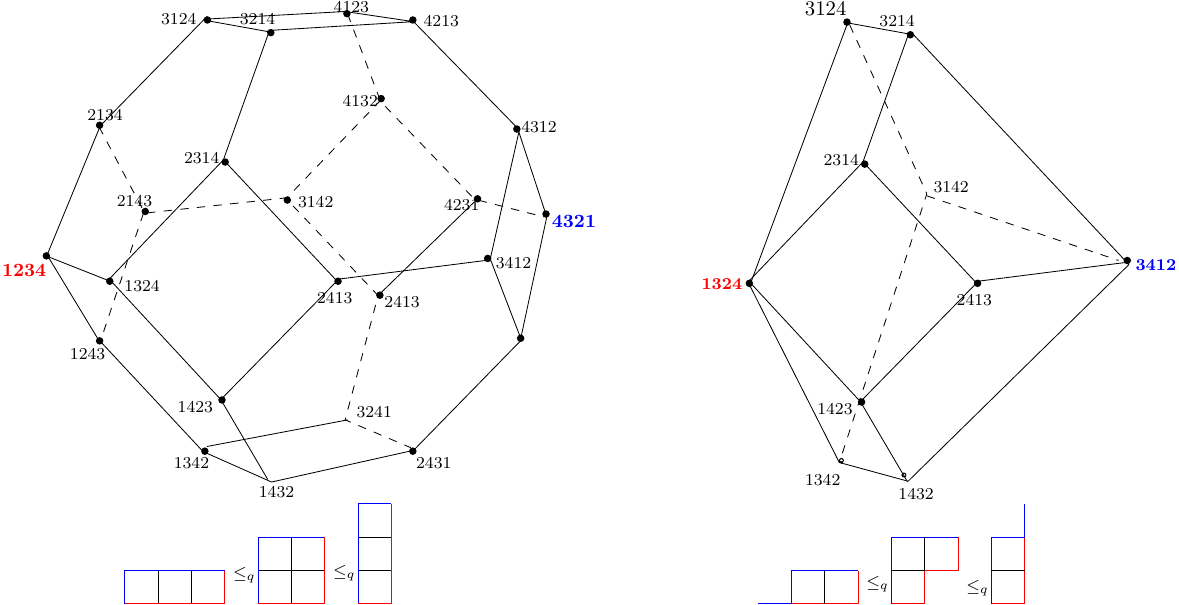}
    \caption{Left: $\Pi_4$ obtained as the uniform flag. Right: an LPFM polytope with two of its vertices highlited.}\label{fig:flag_example_polytope}
\end{figure}

\begin{definition}\label{def:BIP}\cite[Definition A.5]{KodaWill_flag}
Let $u\leq v$ in the Bruhat order on $\S_n$. The \emph{Bruhat interval polytope (BIP) $P_{[u,v]}$} is the polytope given by 
\begin{equation*}
    P_{[u,v]}=\text{conv}\{z=(z_1,\dots,z_n)\in\RR^n: u\leq z\leq v\}.
\end{equation*}
\end{definition}
Given an LPFM $\F:M_1\leq_q M_2\leq_q\cdots\leq_q M_n$ where $M_i=M[U_i,L_i]$ for each $i\in[n]$, it is shown in \cite[Cor. 33]{BenKna} that its flag matroid polytope $P_{\F}$ is in fact a BIP. Namely, $P_{\F}=\text{conv}\{z\in \S_n:\tau_{\mathcal L}\leq z\leq \tau_{\mathcal U} \}$ where $\tau_{\mathcal L}$ is the Bruhat permutation of $L_1\subset\cdots\subset L_n$ and $\tau_{\mathcal U}$ is the Bruhat permutation of $U_1\subset\cdots\subset U_n$. We will refer to the interval $[\tau_{\mathcal L},\tau_{\mathcal U}]$ as the \emph{Bruhat interval of the LPFM $\F$}.

\begin{example}\label{ex:BIP}
    Going back again to the LPFMs in Example~\ref{ex:flag_polytope} we conclude that the polytopes in Figure~\ref{fig:flag_example_polytope} are BIPs. Indeed, the interval $[1324,3412]$ is such that $P_{[1324,3412]}$ is the polytope of the LPFM in $(ii)$. On the other hand, $\Pi_n$ is the BIP corresponding to the interval $[e,\omega]$ where $e=12\cdots n$ and $\omega=n\cdots 21$ are the indentity and the longest permutation in $\S_n$, respectively.
\end{example}

\begin{remark}\label{rem:BIPS_nonnegflag}
    In ~\cite[Proposition 2.9]{TsuWil15} the authors show that every BIP is (the 1-skeleton of) a flag matroid polytope $\M$ that is representable by a point $A\in \F l_n^{\geq 0}$. Conversely, every flag matroid $\M$ that can be represented by some $A\in \F l_n^{\geq 0}$ is such that its flag matroid polytope has all of its vertices in a Bruhat interval $[u,v]$, and thus $P_{\M}$ can be thought of as a BIP. In other words, ~\cite[Proposition 2.9]{TsuWil15} characterizes (full) positroids flags that are representable as points in $\F l_n^{\geq 0}$. As mentioned before, in \cite[Corollary 33]{BenKna} it is shown that (the flag positroid polytope of) every LPFM is a BIP. Hence, LPFMs are representable as points in $\F l_n^{\geq 0}$.
\end{remark}

Let us provide then an instance of a flag of positroids whose flag matroid polytope is not a BIP. Let $\M:M_1\leq M_2\leq_q M_3$ in $[3]$ where $M_1=\{1,3\}$, $M_2=\{12,23\}$ and $M_3=U_{3,3}$ then
$$
P_{\M}=\text{conv}\{(1,2,3),(3,2,1)\}$$
is a line segment which is not an interval in the Bruhat order. Hence the flag of positroids $\M$ is not representable as a point in $\F l_3^{\geq 0}$.

Given a BIP $P_{[u,v]}$ we will abuse terminology and identify it with the interval $[u,v]$ as well as with the flag matroid it comes from. 

The next theorem is a reformulation from \cite{JLLO}, and it also appears in \cite{Borovik2003}. We encourage the reader to consult \cite{PostReinerWill} as a guide to deformations of the permutahedron that give rise to \emph{generalized permutahedrons}. We will not use that terminology here but it might be useful. In particular, every BIP is a generalized permutahedron.

\begin{theorem}\cite[Theorem 3.1]{JLLO}\label{thm:subpermutahedrons} A polytope is the flag matroid polytope of a full flag matroid if and only if its vertices are vertices of $\Pi_n$ and its normal fan is refined by the normal fan of $\Pi_n$.   
\end{theorem}

With this theorem in mind we will now move on to the section presenting our main contributions. In particular, we will consider subdivions of $\Pi_n$ into two pieces such that each piece satisfies Theorem~\ref{thm:subpermutahedrons}. This will be our fist step into understanding subdivisions of $\Pi_n$ into BIPs. 

\section{Main Results}\label{sec:main_results}

The identity permutation of $\S_n$ is denoted by $e$ and written in one-line notation as $e=123\dots n$. The longest permutation is denoted $\omega$ and written in one-line notation as $\omega=n(n-1)\cdots 21$. The reader must remember that $e$ (and $\omega$) depends on $n$ although it does not appear in their notation. Let $A=\{i_1<\cdots<i_m\}\subseteq [n]$ and define the following sequences:
\begin{enumerate}
    \item[$\bullet$] The \emph{increasing sequence of the set $A$ is} $\overrightarrow A:=i_1i_2\cdots i_m$.
    \item[$\bullet$] The \emph{decreasing sequence of the set $A$ is} $\overleftarrow A:=i_mi_{m-1}\cdots i_1$.
    \item[$\bullet$] The sequence $\omega_A\in\S_n$ corresponds to deleting in the longest permutation $\omega$ the values of the set $A$. 
    \item[$\bullet$] The sequence $e_A$ corresponds to deleting in the identity permutation $e\in\S_n$ the values of the set $A$. 
\end{enumerate}
 For instance, if $n=7$ and $A=\{3,5,6\}$ then $\overrightarrow A=356, \overleftarrow A=653, \omega_A=7421, e_A=1247$. Also, $\overrightarrow Ae_A$ denotes the concatenation of the sequences $\overrightarrow A$ and $e_A$. Similar definitions are made for $\overleftarrow A$ and $\omega_A$. Continuing with our example, $\overrightarrow Ae_A=3561247$ and $\omega_A\overleftarrow A=7421653$.

We will make use of this notation in our main results which characterize how to subdivide in a coarse way the polytope $\Pi_n=P_{\U_n}$ into pieces corresponding to BIPs coming from LPFMs. This will be done using hyperplanes.

If $H$ is a hyperplane in $\RR^n$ with affine equation $\sum_{i=1}^na_ix_i=b$, 
we set $H^{>}:=\{p \in \RR^n:\sum_{i=1}^na_ip_i>b\}$, and $H^{<}:=\{p \in \RR^n:\sum_{i=1}^na_ip_i<b\}$. Similar definitions are made for $H^{\geq}$ and $H^{\leq}$.
 
\begin{definition}\label{def:hyper_split}\cite{JosSchro}
A \emph{split} of a polytope $P$ is a polytopal subdivision $\Sigma$ of $P$ (without new vertices) with exactly two maximal cells $F_1$, $F_2$. The intersection $F_1\cap F_2$ is a codimension 1 cell and its affine span is the \emph{split hyperplane $H$} of the subdivision $\Sigma$. In this case we refer to $\Sigma$ as the $H$-split of $P$ and the cells $F_1$ and $F_2$ correspond to  $P\cap H^{\geq}$ and $P\cap H^{\leq}$.
\end{definition}

In this paper we are interested in splits of $\Pi_n$ such that $F_1$ and $F_2$ are positroid flags, and more specifically LPFMs. We will study these splits of $\Pi_n$ by describing the corresponding split hyperplanes. Using the notation in Definition \ref{def:hyper_split}, we observe that since $F_1\cap F_2$ is a facet of the polytopes $F_1$ and $F_2$ then the normal to this facet must be normal to a hyperplane from Equation \ref{eq:Pi_n} (see \cite{PostReinerWill}).
Therefore we will describe the hyperplanes $H$ that are parallel translations of those in Equation \ref{eq:Pi_n} and give rise to an $H$-split of $\Pi_n$ where $F_1$ and $F_2$ are LPFMs. In Figure \ref{fig:poset_4_subdiv} below we exemplify the $H$ splits of $Perm_4$ as minimal elements of a poset. 

\begin{definition}\label{def:schubert}
    An LPM $M=M[U,L]$ of rank $k$ on $[n]$ is said to be a \emph{Schubert LPM} if $L=\{n,n-1,\dots,n-k+1\}$. Similarly, $M$ is called a \emph{dual Schubert LPM} if $U=\{1,2,\dots,k\}$.
\end{definition}
This terminology comes from the fact that \emph{Schubert matroids} are generic points in the cells of the Schubert decomposition of $Gr_{k,n}$ (see~\cite{KodaWill_gras}).

The following lemma tells us how to characterize flags of LPMs whose constituents are all Schubert (or dual Schubert).

\begin{lemma}\label{lemma:schub_is_quo}
   Let $M=M[U,L]$ and $M'=M[U',L']$ be two Schubert matroids on $[n]$, where $L'=L\setminus\{\ell_1\}$. Then $M'\leq_q M$. Dually, let $M=M[U,L]$ and $M'=M[U',L']$ be two dual Schubert matroids on $[n]$ where $U'=\{1,\dots,k\}$ and $U=U'\setminus\{k\}$. Then $M\leq_q M'$.
\end{lemma}
\begin{proof}
    This follows from the graphical interpretation of being a good pair. Namely, if $U'=U\setminus\{u_j\}$  for any $j\in[k]$, then $\ell_1$ lies in the region $x\geq a, y\leq b$ if $(a,b)$ are the coordinates of the northern corner of the path $u_j$. This is illustrated in Figure~\ref{fig:schubertlpm}, where $\ell_1$ is in bold.
    The second part of the Lemma follows from the first and duality.
    \begin{figure}[htp]
    \centering
    \includegraphics[width=.4\textwidth]{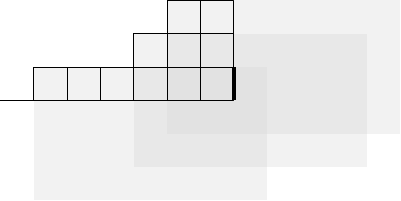}
    \caption{$\ell_1$ is a good pair with either $2, 6, 8$.}\label{fig:schubertlpm}
\end{figure}
\end{proof}

It is known that the 2-dimensional faces of $\Pi_n$ are either hexagonal or quadrilateral. Using \cite[Theorem 4.3]{JLLO}, \cite[Theorem A]{BEW24} one has that a hyperplane $H$ gives a split of $\Pi_n$ into flags of positroids corresponding to BIPs if quadrilaterals do not split, and if a hexagonal face splits then the minimal and maximal permutations of that face are in distinct sides of $H$. 

\begin{definition}\label{def:good_bad_split}
    An $H$-split of $\Pi_n$ is said to be \emph{good} if two dimensional faces are subdivided into BIPs. Otherwise, we say the $H$-split is \emph{bad}.
\end{definition}

We illustrate the two types of bad splits in Figure~\ref{fig:goodbad}. The one via the hyperplane $H:x_1+x_2=5$ splits a square face. The other one is such that, for instance, in the hexagon given by the interval $[1234, 1432]$ the minimal and maximal permutation are on the hyperplane, not in different sides of it. The reader is invited to check that, for instance, the six permutations in the hexagon in Figure~\ref{fig:goodbad} is a flag of positroids with constituents $\{1,2,4\}, \{13,23,34\},\{123,134,234\}$ although this flag is not realizable as a point in $\F l_4^{\geq 0}$ as those six permutations are not an interval.

    \begin{figure}[htp]
    \centering
    \includegraphics[width=.9\textwidth]{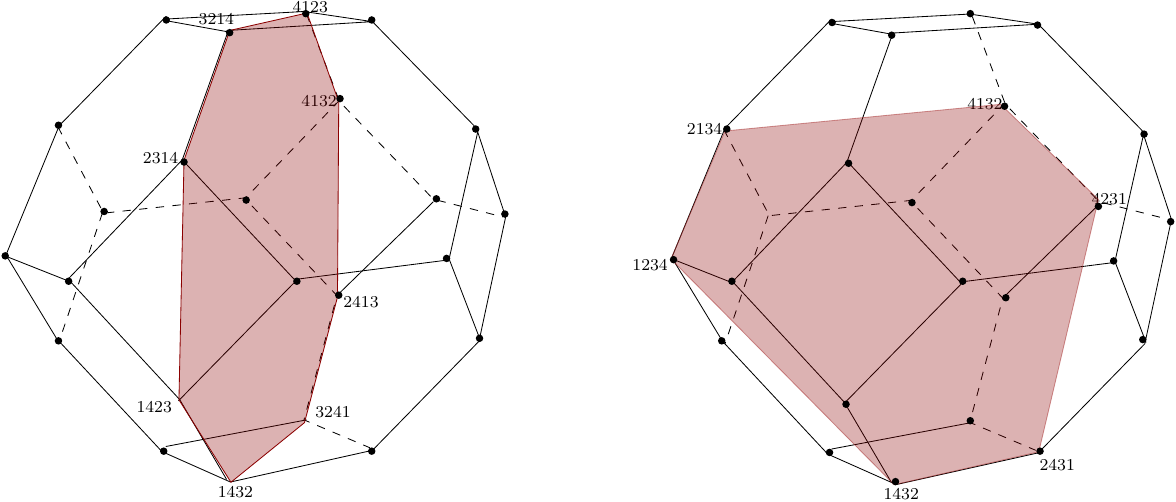}
    \caption{Left: Bad $H$-split with $H:x_1+x_2=5$. Right: Bad $H$-split with $H:x_3=3$}\label{fig:goodbad}
\end{figure}


\begin{proposition}\label{prop:coarse_is_schubert}
    Consider a split of $\Pi_n$ with maximal cells $F_1,F_2$. If $F_1$ and $F_2$ are Bruhat interval polytopes then they are LPFMs. 
\end{proposition}
\begin{proof}
   Since $F_1$ and $F_2$ 
   correspond to two intervals whose union is $\Pi_n$, we can assume without loss of generality that $e\in F_1$, and hence $\omega\in F_2$. That is the Bruhat interval of $F_1$ and $F_2$ are of the form $[e,v]$ and $[u,\omega]$, respectively, for some $u,v\in\S_n$.
  It now suffices to notice that the constituents of (the flag matroid of) $F_1$, say $M_1\leq_q M_2\leq_q\cdots\leq_q M_n$ are such that the $L$-basis of $M_i$ is $\{1,2,\cdots,i\}$ since there is a minimal flag of basis whose indicator vector is $e$. On the other hand, there is a unique maximal flag of bases $U_1\subset U_2\subset\cdots\subset U_n$ whose indicator vector is $v$ (see \cite[Section 1.7]{Borovik2003}). Therefore, for each $i$ we have that $M_i$ is a Schubert matroid. Similarly, the constituents of $F_2$ are dual Schubert matroids. The result follows.
\end{proof}

\begin{remark}\label{rem:coarse_isLPM}
    From the proof in Proposition~\ref{prop:coarse_is_schubert}, we concluded that $F_1$, has as 1-skeleton the interval $[e,v]$ for some $v\in S_n$. On the other hand, the 1-skeleton of $F_2$ is the interval $[u,\omega]$ for some $u\in S_n$. We will unveil shortly what are the $u,v$ that give rise to such subdivision and conclude that that $u\leq v$.
\end{remark}

\begin{definition}\label{def:dual_perm}
    Let $\tau\in S_n$. The \emph{permutation dual} of $\tau$ is the permutation $\tau^*\in S_n$ such that $\tau^*(j):=n-\tau(j)+1$, for every $j\in[n]$.
\end{definition}

Let $\F: M_1\leq_q M_2\leq_q\cdots\leq_q M_n$ be an LPFM on $[n]$ and let $[\tau_{\mathcal L},\tau_{\mathcal U}]$ be the corresponding Bruhat interval. Let $\F^*:M_n^*\leq_q M_{n-1}^*\leq_q\cdots\leq_q M_1^*$ be the LPFM obtained from $\F$ by dualizing its constituents. Using Definition \ref{def:dual_perm} it is not difficult to verify that the Bruhat interval polytope of the LPFM $\F^*$ corresponds to the interval $[\sigma,\tau]^*:=[\tau^*, \sigma^*]$. We call $[\sigma,\tau]^*$ the \emph{dual of the interval $[\sigma,\tau]$}.

\begin{example}\label{ex:dual_perm} In Figure~\ref{fig:dual_op} we illustrate that 
    $[e,316542]^*=[461235,\omega]$. We encourage the reader to check that $[132456,\omega]^*=[e,645321]$ and to illustrate the corresponding flags.
      \begin{figure}[htp]
    \centering
    \includegraphics[width=.9\textwidth]{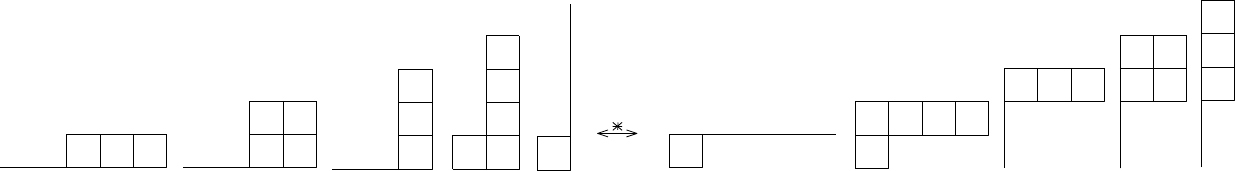}
    \caption{Left: Flag $\F$ corresponding to $[e,316542]$. Right: dual of $\F$ corresponding to $[461235,\omega]$}\label{fig:dual_op}
\end{figure}  
\end{example}

Elaborating on Remark~\ref{rem:coarse_isLPM} it follows that $[e,v]^*=[v^*,\omega]$ and $[u,\omega]^*=[e,u^*]$. Letting $H$ be the affine span of $F_1\cap F_2$, we define the \emph{dual of the split hyperplane $H$}, or simply, \emph{the dual of the $H$-split} as the hyperplane $H^*$-split of $\Pi_n$ where $H^*$ is the affine span of $F_1^*\cap F_2^*$ where $F_1^*,F_2^*$ are the polytopes whose 1-skeleton is $[e,u^*]$ and $[v^*,\omega]$, respectively. In other words, a coarse (non-trivial) subdivision of $\Pi_n$ gives rise to its dual subdivision, by dualizing either the corresponding LPFMs, or dualizing the corresponding Bruhat intervals of the maximal cells.

Now we are ready to state and prove one of our main results.

\begin{theorem}\label{thm:coarsest_v1}
    Let $n\geq 3$. Then each of the following hyperplanes 
    \begin{enumerate}
    \item[(T1)] $H:x_1+x_2+\cdots+x_j=1+2+\cdots+(j-1)+(j+1)={{j+1}\choose 2}+1$ for $j\leq n-2$.
    \item[(T2)] $H:x_1+x_2+\cdots+x_j=n+(n-1)+\cdots+(n-(j-2))+(n-j)$ for $j\leq n-2$.
    \item[(T3)] $H:x_i=r$ where $2\leq r\leq n-1$ and $i\in\{1,n\}$. 
    \end{enumerate}
give rise to a split of $\Pi_n$ into BIPs, and hence into LPFMs. 
\end{theorem}

\begin{proof} The strategy is to show that in each case, the given hyperplane split is a decomposition of $\Pi_n$ into two LPFMs $F_1$, $F_2$. Thus, in view of \cite[Cor. 33]{BenKna} the flag polytopes $P_{F_1}$ and $P_{F_2}$ are Bruhat interval polytopes. We intend to show what are the corresponding Bruhat intervals.

(T1): Let $H:x_1+x_2+\cdots+x_j=1+2+\cdots+(j-1)+(j+1)$ where $j\leq n-2$ is fixed and let $\sigma=b_1b_2\cdots b_n\in S_n$ in one-line notation. Then
    \[\begin{cases}
    \sigma\in H^{>} & \text{ if } b_1+\cdots+b_j>{{j+1}\choose 2}+1,\\
    \sigma\in H^{<} & \text{ if } b_1+\cdots+b_j={{j+1}\choose 2},\\
    \sigma\in H & \text{ if } b_1+\cdots+b_j={{j+1}\choose 2}+1
\end{cases}\Longleftrightarrow
\begin{cases}
    \sigma\in H^{>} & \text{ if } b_1+\cdots+b_j>{{j+1}\choose 2}+1,\\
    \sigma\in H^{<} & \text{ if } \{b_1,\ldots,b_j\}=[j],\\
    \sigma\in H & \text{ if }  \{b_1,\ldots,b_j\}=[j-1]\cup\{j+1\}.
\end{cases}\]  
Now let $A=\{1,\dots,j-1,j+1\}$ and notice that notice that $H\cap \Pi_n=[{\overrightarrow A}e_A\, , {\overleftarrow A}\omega_A]$. If $\sigma\in H^{<}$ then $\sigma\lessdot(j,j+1)\sigma\leq\overleftarrow{A}\omega_A$ and thus $\Pi_n\cap H^{\leq}=[e,{\overleftarrow A}\omega_A]$. 

Let us show now that $\Pi_n\cap H^{\geq}=[{\overrightarrow A}e_A,\omega]$. If $\sigma\in H^{\geq n}$ set $X_{\sigma}:=\{b_1,\ldots,b_j\}\neq[j]$, set $Y_{\sigma}:=[n]\setminus X$ and notice that the permutation $\tau$ obtained by concatenating $\overrightarrow X_{\sigma}$ and $\overrightarrow Y_{\sigma}$, that is, $\tau=\overrightarrow{X_{\sigma}}\cdot\overrightarrow{Y_{\sigma}}$ satisfies that $\tau\leq\sigma$. Now, we will construct a chain in the Bruhat order that will lead us to conclude that $\overrightarrow A\cdot e_A\leq\overrightarrow {X_{\sigma}}\cdot \overrightarrow {Y_{\sigma}}\leq \sigma$ and the result will follow. Notice that $A=(A\cap X_{\sigma})\sqcup(A\cap Y_{\sigma})$ and $A\cap X_{\sigma}$ appears in the first $k$ positions of $\overrightarrow {X_{\sigma}}$, for some $k\leq j$. Similarly, $A\cap Y_{\sigma}$ appears in the first $l$ positions of $\overrightarrow{Y_{\sigma}}$, for some $l$. Then in $\tau$ the set $A\cup \{j\}$ appears in positions $[k]\cup\{j+1,\ldots,j+l\}$. 

Thus $\tau_1:=\tau(k+1,j+1)$ satisfies $\tau_1\lessdot\tau$ and $|A\cap X_{\tau_1}|=1+|A\cap X_{\tau}|$. Then it follows that the permutation $\tau_1'$ obtained from $\tau$ after ordering increasingly the values in entries $\{j+1,\ldots,n\}$ is such that $\tau_1'\leq\tau_1$. Similarly, the permutation $\tau_2:=\tau_1'(k+2,j+1)$ satisfies $\tau_2\lessdot\tau_1'$ and $|A\cap X_{\tau_2}|=1+|A\cap X_{\tau_1'}|$. In this fashion we arrive to a sequence of permutations $(\tau_{j-k},\tau_{j-k-1}',\tau_{j-k-2},\tau_{j-k-2}',\ldots,\tau_1,\tau)$ where $X_{\tau_{j-k}}=A$ and  
$$\overrightarrow Ae_A\leq \tau_{j-k}\lessdot\tau_{j-k-1}'\leq\tau_{j-k-1}\lessdot\tau_{j-k-2}\leq\cdots\tau_2\lessdot\tau_1'\leq\tau_1\lessdot\tau\leq \sigma.$$

It follows that the $H$-hyperplane split of $\Pi_n$ gives rise to the subdivision $[e,{\overleftarrow A}\omega_A]\cup[{\overrightarrow A}e_A,\omega]$, where each interval is a flag of Schubert and dual Schubert, respectively.

(T2): In (T1) we obtained the split subdivision $[e,{\overleftarrow A}\omega_A]\cup[{\overrightarrow A}e_A,\omega$]. Using Definition~\ref{def:dual_perm} the reader can verify thta  
$[e,{\overleftarrow A}\omega_A]^*=[{\overrightarrow B}e_B,\omega]$ and $[{\overrightarrow A}e_A,\omega]^*=[e,{\overleftarrow B}\omega_B]$ where $B=\{n-j, n-j+2,n-j+3,\dots,n-1,n\}$. Since this subdivision is also into LPFMs, in view of Lemma~\ref{lem: dual_split}, and moreover it is given by the hyperplane in (T2), the result follows.

(T3): Fix $r\in\{2,\dots,n-1\}$ and consider the hyperplane $H:x_1=r$. A permutation $\sigma=a_1a_2\cdots a_n\in S_n$ is such that 
    $$\begin{cases}
    \sigma\in H^{\geq} & \text{ if } a_1\geq r,\\
    \sigma\in H^{\leq} & \text{ if } a_2\leq r
\end{cases}$$
 and hence the $H$-split subdivision of $\Pi_n$ gives rise to the intervals $[e,r{\omega_r}]\cup[r{e_r}, \omega]$. This follows since for any $k\leq r$ it holds that: (i) $\tau\leq k{\omega}_k$ for any $\tau\in S_n$ such that $\tau(1)=k$; (ii) $(k,r)(kn(n-1)\cdots (k+1)(k-1)\cdots 1)=rn(n-1)\cdots (r+1)(r-1)\cdots 1$. That is, $(k,r)\cdot k{\omega}_k=r{\omega}_r\Longleftrightarrow k{\omega}_k\leq r{\omega}_r$ and thus $H^{\leq}\cap \Pi_n =[e,r{\omega}_r]$. In a similar fashion $H^{\geq}\cap \Pi_n =[r {e}_r,\omega]$.
 
 Now, the Bruhat interval $[e,r {\omega}_r]$ is the (1-skeleton of the) polytope flag whose constituents are Schubert matroids and thus, by Lemma \ref{lemma:schub_is_quo}, the given chain is a flag of LPMs. On the other hand, the interval $[r {e}_r,\omega]$ corresponds to a flag of dual Schubert LPMs. 

The case for $H:x_n=r$ has a similar analysis. In this case the $H$-split subdivision of $\Pi_n$ is the subdivision into the intervals $[e,{\omega}_rr]\cup[{e}_rr, \omega]$.  This allows us to conclude the result
\end{proof}

The following Lemma whose proof follows directly from the definitions and the proof of Theorem \ref{thm:coarsest_v1} is in order.
\begin{lemma}\label{lem: dual_split}
One has that 
\[\{H^*:H\text{ is a hyperplane of type }(T1)\}=\{G:G\text{ is a hyperplane of type }(T2)\}.\]
Additionally, a hyperplane $x_1=r$ of type (T3) is dual to $x_1=n-r+1$ of the same type.
\end{lemma}

In Example \ref{ex:poset_L4} we illustrate Theorem \ref{thm:coarsest_v1} and Lemma \ref{lem: dual_split}.


In Theorem \ref{thm:coarsest_v1} we proved that  (T1), (T2) and (T3) give rise to $H$-splits of $\Pi_n$ into pieces that are LPFMs. The next main result tells us that any $H$-split of $\Pi_n$ into BIPs is of either of those 3 types, hence we characterize them all. 

\begin{theorem}\label{thm: all_split_hyperplanes}
The only hyperplane splits of $\Pi_n$ that give rise to a split into BIPs are the ones described in conditions (T1), (T2), (T3). 
\end{theorem}
\begin{proof}
The idea of the proof will be to show that hyperplanes described in Theorem \ref{thm:coarsest_v1} are the only parallel translations of facet-defining hyperplanes on $\Pi_n$ that do not split $2$-dimensional faces (squares and hexagons) in a bad way as described in Definition~\ref{def:good_bad_split}.  
To this end, let $S=\{i_1<i_2<\cdots<i_j\}$ and consider $\alpha=\alpha_1+\cdots+\alpha_j$ such that ${{j+1}\choose 2}<\alpha<n+(n-1)+\cdots+(n-j+1)$.
    Let $|S|>1$ and let $H:x_{i_1}+x_{i_2}+\cdots+x_{i_j}=\alpha_1+\alpha_2+\cdots+\alpha_j$.

    CASE I: Suppose there exist indices $1<a<b<n$ such that 
    $$\alpha_{a-1}<\alpha_a-1<\alpha_a<\alpha_b<\alpha_b+1\leq\alpha_{b+1}$$ then $s_{\alpha_a-1}$ and $s_{\alpha_b}$ commute.
    \begin{enumerate}
        \item[Ia] Assume there exists $\ell,k$ such that $1\leq\ell\leq a$ and $b\leq k\leq j$ with $i_{\ell-1}+1<i_{\ell}$, $i_k+1<i_{k+1}$. This condition is about having enough room to build a permutation $\sigma$ as follows:
        $\sigma(i_r)=\alpha_r$ for each $r\in[j]$, $\sigma(i_{\ell}-1)=\alpha_a-1$, $\sigma(i_k+1)=\alpha_b+1$. The remaining positions of $\sigma$ are filled with the unused entries. Then $s_{\alpha_a-1}s_{\alpha_b}\sigma=s_{\alpha_b}s_{\alpha_a-1}\sigma\in H$ although $s_{\alpha_a-1}\sigma\in H^{<}$, $s_{\alpha_b}\sigma\in H^{>}$.
        \item[Ib] If only one of $k,\ell$ as above exists, say only $k$, then take $\sigma$ such that $\sigma(i_r)=\alpha_r$ and place $\alpha_a-1,\alpha_b+1$, in that order, to the right of $\alpha_b$. Fill the remaining positions with the remaining values. Then $\sigma$ and $s_{\alpha_a-1}s_{\alpha_b}\sigma\in H$ but  $s_{\alpha_a-1}\sigma\in H^{<}$, $s_{\alpha_b}\sigma\in H^{>}$
        \item[Ic] Neither $k$ nor $\ell$ exists. Take $\sigma$ such that $\sigma(i_r)=\alpha_r$ for $r\neq a, r\neq b$, $\sigma(i_a)=\alpha_b, \sigma(i_b)=\alpha_a$. In between $\alpha_b, \alpha_a$ place, in that order, $\alpha_a-1, \alpha_b+1$. The remaining values are filled in the remaining positions. Then $\sigma,s_{\alpha_a-1}s_{\alpha_b}\sigma\in H$ but $s_{\alpha_a-1}\sigma\in H^<, s_{\alpha_b}\sigma\in H^{>}$.
    \end{enumerate}
    Notice that the condition given in CASE I guarantees that the value $\alpha$ can be written in more than one way under the stated conditions. Now it remains to consider the case in which $\alpha$ can only be written in one way. 
    
    CASE II: Suppose $\alpha_1+\alpha_2+\cdots+\alpha_j=1+2+\cdots+(j-1)+(j+1)$ or $n+(n-1)+\cdots+(n-j+2)+(n-j)$. It is enough to analyze the former situation. 
    \begin{enumerate}
        \item[IIa] Let $\alpha_r=r$ for $r\leq j-1$, $\alpha_j=j+1$ and consider 
        $\sigma(i_r)=\alpha_r$ for $r\in[j]$. Suppose there are positions $k<i_{j-1}$ and $\ell>i_j$ and place $j,j+2$, respectively. In this way, the permutation $\sigma$ has values $j, j+1, j+2$ in that order. Applying the braid relation $s_js_{j+1}s_j$ to sigma yields $s_js_{j+1}s_j\sigma\in H$. On the other hand, $s_j\sigma$ and $s_{j+1}s_j\sigma$ are in $H^<$. Similarly, $s_{j+1}\sigma$ and $s_js_{j+1}\sigma$ are in $H^>$.
        \item[IIb] If only one of $k,\ell$ is not available, say $k$ then the hyperplane in this case is of the form $x_1+x_2+\cdots+x_{j-1}+x_{j'}=\alpha$ for some $j'\geq j$. If $j'<j$ let $\sigma(j)=j$ and place $j+2$ to the right of $j$. Applying again $s_js_{j+1}s_j$ to $\sigma$ leads to the analysis above. If $j'=j$ then this is a good split: $[e,\overleftarrow{S}\!\omega_S]\cup[\overrightarrow{S}e_S,\omega]$.  
    \end{enumerate}

    Finally, if $|S|=1$ the hyperplane $H:x_i=\alpha$, where $2\leq\alpha\leq n-1$ is such that taking $\sigma\in S_n$ whose values in positions $i-1,i,i+1$ are $\alpha-1,\alpha, \alpha+1$ in that order, satisfies that $H$ gives a bad split of the interval $[\sigma,s_{\alpha}s_{\alpha-1}s_{\alpha}\sigma]$. Notice that to construct such $\sigma$, the only constrain is for $i\neq 1, i\neq n$, in accordance with (T3) in Theorem~\ref{thm:coarsest_v1}. The reader is encourage to look back again at Figure~\ref{fig:goodbad} to determine which case do they belong to that makes them bad.
    \end{proof}

    \begin{remark}\label{rem:coarsest_subdiv}
In~\cite[Table 2]{BEW24} the authors display all the 9 coarsest nontrivial subdivisions of $\Pi_4$ into BIPS. Three of them are subdivisions into four BIPS, whereas the ones obtained by us are into two BIPs. We will go back to this comment in our last Section.
    \end{remark}

Part of the motivation for this paper was to understand the finest subdivisions of $\Pi_n$ into BIPs coming from LPFMs. That is, pieces that correspond to polytopes of full flag lattice path matroids. Since Theorems~\ref{thm:coarsest_v1} and ~\ref{thm: all_split_hyperplanes} provide an answer to this for the coarsest ones, rather than the finest ones, we now turn our attention to refinements of these subdivisions.

\begin{definition}\label{def:poset_refinement}
    Denote by $\mathfrak L_n$ the poset whose elements are the polytopal subdivisions of $\Pi_n$ obtained by simultaneously splitting it using finitely many hyperplanes of types (T1), (T2) or (T3), and ordered by refinement. If $\Sigma\in\mathfrak L_n$ is obtained via the hyperplanes $H_1,\dots,H_m$ we write $\Sigma=H_1\wedge\cdots\wedge H_m$.
\end{definition}

\begin{example}\label{ex:poset_L4}
    Let us illustrate all the non-trivial subdivisions of $\Pi_4$ into BIPs coming from LPFMs, as well as the elements of the poset $\mathfrak L_4$.
    Using Theorems \ref{thm:coarsest_v1} and \ref{thm: all_split_hyperplanes} we obtain a complete list of the six hyperplane splits of $Perm_4$:
    \begin{enumerate}
        \item[(T1)] $H_1: x_1+x_2=4$
        \item[(T2)] $H_2: x_1+x_2=6$
        \item[(T3)] $H_3: x_1=2$\quad $H_4:x_1=3$\quad $H_5:x_4=2$\quad $H_6:x_4=3$.
    \end{enumerate}\label{ex:splits_p4}
 Now, Lemma \ref{lem: dual_split} tells us that $H_1^*=H_2$ and $H_3^*=H_4, H_5^*=H_6$.
 In Figure \ref{fig:poset_4_subdiv} we illustrate the poset $\mathfrak L_4$. There we write $x_{12}$ meaning $x_1+x_2$. For instance, the common refinement $H_2\wedge H_3\wedge H_6$ provides the subdivision in the top-right into the intervals
 $$
 [e,2413]\cup[1243,2431]\cup[2134,4213]\cup[2143,4231]\cup[2413,\omega].
 $$
 
  \begin{figure}[ht]
    \centering
    \includegraphics[width=1\linewidth]{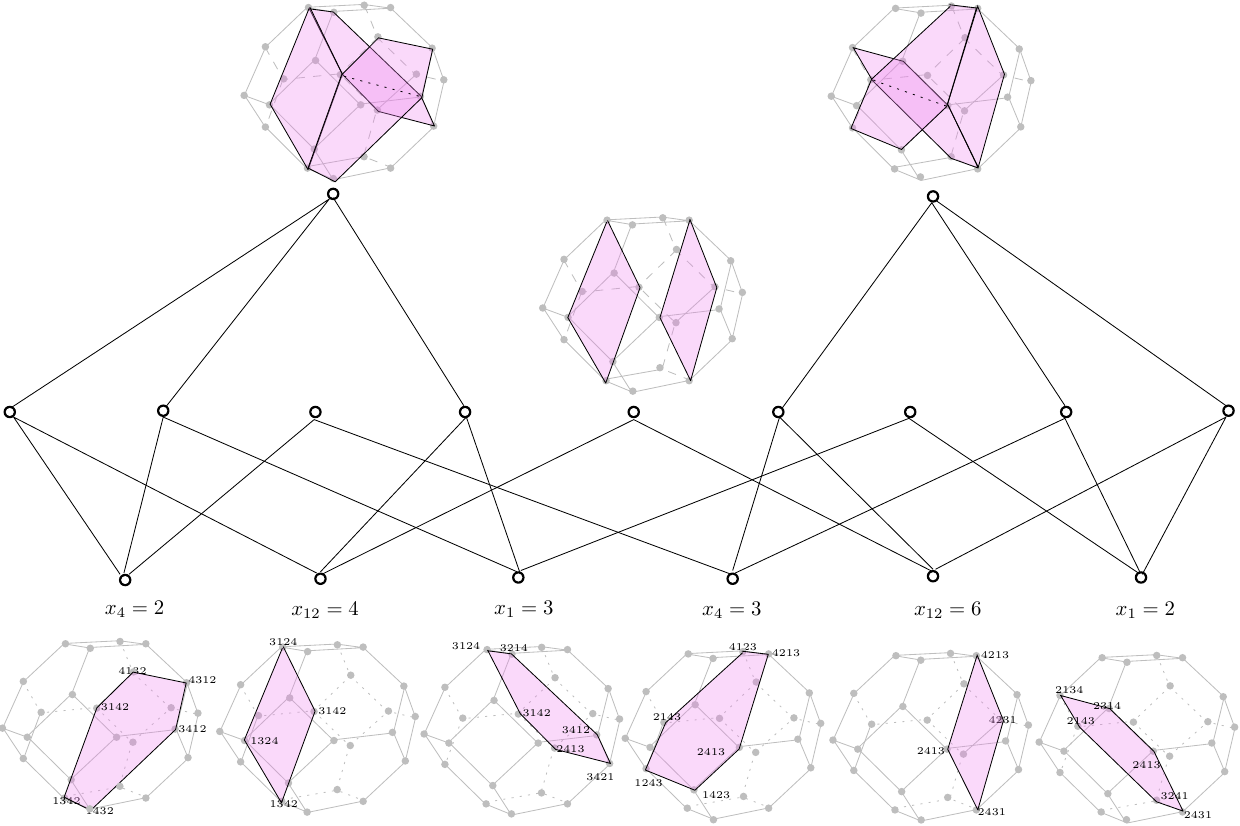}
    \caption{Poset $\mathfrak L_4$ of split subdivisions of $\Pi_4$ by refinement.}
    \label{fig:poset_4_subdiv}
\end{figure} 
In \cite[Table 1]{BEW24} the authors lists all the 14 finest (regular) subdivisions of $\Pi_4$ into Bruhat Interval Polytopes. Out of those 14 finest subdivisions, the bottom two in that Table are the ones we obtain here as the top elements of $\mathfrak L_4$ in Figure \ref{fig:poset_4_subdiv}. On the other hand, the six coarsest subdivisions mentioned in Remark~\ref{rem:coarse_isLPM} that appear in
\cite[Table 2]{BEW24} are the minimal elements of the poset $\mathfrak L_4$.
\end{example}

In the next section we will pose some questions and conjectures regarding the analysis made so far, along with the connection to tropical geometry.

\section{Further questions}\label{sec:problems}

In Example~\ref{ex:poset_L4} we manage to display the whole poset $\mathfrak L_4$, thanks to the understanding of the coarsests subdivisions of $\Pi_4$ into LPFMs, and using SAGE. However, we do not know in general what are all the elements of $\mathfrak L_n$. Thus we start by posing the following. 

\begin{problem}
    Describe all the subdivisions of $\Pi_n$ into BIPs coming from LPFMs. That is, describe the poset $\mathfrak L_n$. In particular, describe the finest of such subdivisions.
\end{problem}

This problem requires one to understand the issues that may arise when subdividing $\Pi_n$ with different hyperplanes of types (T1), (T2), (T3). One possible issue is for the refinement to create new vertices. For instance, using the notation in Example \ref{ex:poset_L4} one has that $H_2\wedge H_5$ is not an element of $\mathfrak L_4$ as it creates new vertices (see Figure~\ref{fig:H2H5bad}).

  \begin{figure}[ht]
    \centering
    \includegraphics[width=0.5\linewidth]{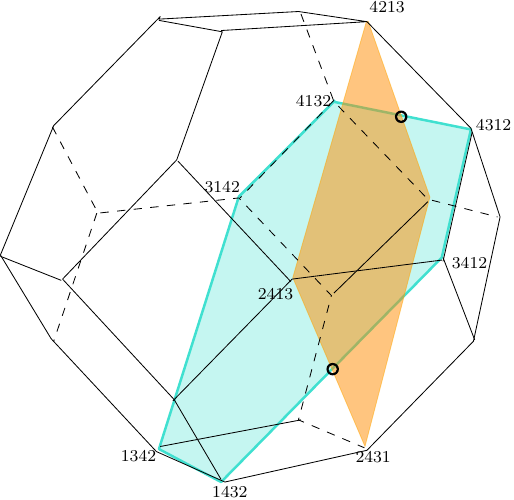}
    \caption{A refinement of two minimal elements of $\mathcal L_4$ that is not in $\mathcal L_4$.}
    \label{fig:H2H5bad}
\end{figure}

Nonetheless, an observation regarding this problem is that some obvious common refinements come from hyperplanes with the same normals and different level sets. For instance, in $\mathfrak L_5$ a subdivision comes from the common refinement of the hyperplanes $x_1=2, x_1=3,x_1=4$. On the other hand, notice from Figure~\ref{fig:poset_4_subdiv} that $H_3\wedge H_6\in\mathfrak L_4$ although those two hyperplanes intersect in their interior (inside $\Pi_4$). Thus not every element in $\mathfrak L_n$ come from non-crossing hyperplanes. 

Changing gears a little bit, another motivation for wanting to understand the subdivisions we obtained in this paper comes from tropical geometry. Given a polytope $P\subset\RR^n$ with vertices $v_1,\dots,v_m$ and a subdivision $\Delta$ of $P$ we say that $\Delta$ is \emph{regular (coherent)} if there exist a tuple of reals $(\mu_1,\dots,\mu_m)$ such that $\Delta$ is obtained by projecting the lower facets of $\text{conv}\{(v_1,\mu_1),\dots,(v_m,\mu_m) \}$ back to $\RR^n$. 

Subdivisions of $\Pi_n$ into BIPs can be parametrized by cones in the \emph{positive tropical complete flag variety} $Tr^{>0}\F l_n$ (see~\cite{Bor22,JLLO,BEW24}). We will not delve into the exciting area of tropical geometry here, but we will content ourselves with stating the following.

\begin{lemma}\label{lem:lpfms_tropical}
    Each of the subdivions of $\Pi_n$ given by hyperplanes of types (T1), (T2), (T3) is regular. 
\end{lemma}
\begin{proof}
    This follows from the fact that conditions (T1), (T2), (T3) provide subdivisions of $\Pi_n$ into two BIPs.
\end{proof}

In~\cite{BEW24} the authors provide different fan structures of $Tr^{>0}\F l_n$. The one that gives rise to ~\cite[Table 1, Table 2]{BEW24} is a fan structure that, as they mention there, is dual to the 3-dimensional associahedron. Hence, the $f$-vector of this fan is $(14, 21, 9, 1)$. Column 1 of ~\cite[Table 1]{BEW24} provide 14 weights each coming from one of the maximal cones of the mentioned fan structure. 
Similarly, column 1 of ~\cite[Table 2]{BEW24} provide 9 weights each coming from one ray of the fan. Already, we accounted for 2 of those 14 and 6 of those 9 cones of dimensions 3 and 1, respectively, as there are two finest and 6 coarsest subdivisions of $\Pi_4$ into LPFMs. 

In Figure~\ref{fig:lpfm_in_fan} we relate the two finests subdivisions of $\Pi_4$ into LPFMs with the two maximal cones in $Tr^{>0}\F l_n$ that they correspond to. We draw it next to the 3-dimensional associahedron as explained here.

Hence the following questions are in order.

\begin{problem}
    Describe the elements of $\mathfrak L_n$ as regular subdivisions of $\Pi_n$, coming from points in $Tr^{>0}\F l_n$.
\end{problem}

\begin{problem}
    How does the subcone of LPFMs sit inside the cone of $Tr^{>0}\F l_n$, with respect to a given fan structure.?
\end{problem}

  \begin{figure}[ht]
    \centering
    \includegraphics[width=0.7\linewidth]{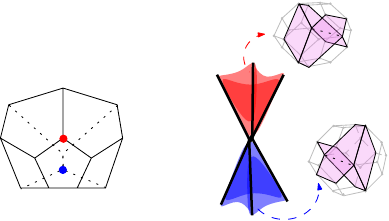}
    \caption{Finest subdivisiones of $\Pi_4$ as points in maximal cones of fan structure of $Tr^{>0}\F l_4$ dual to the associahedron.}
    \label{fig:lpfm_in_fan}
\end{figure}

Our next question has to do with the relation of subdivisions of $\Pi_n$ into LPFMs and matroid minors.
Namely, if $H:x_1=r$ is a hyperplane of type (T3) let $\M_1$ and $\M_2$ be the two LPFMs that the $H$-split of $\Pi_n$ give rise. It is an exercise to the reader to verify that deleting the element $n$ of each constituent of $\M_1$ and $\M_2$ give rise to two flags $\M_1'$ and $\M_2'$ which come from a type (T3) subdivision of $\Pi_{n-1}$ into LPFMs. Hence we pose the following question.

\begin{problem}
    Describe the subdivisions of $\Pi_n$ into LPFMs that can be constructed recursively via matroid operations.
\end{problem}

Another question we are interested in is the following

\begin{problem}
    Does every full dimensional LPFM appear as a cell in a subdivision of $\Pi_n$ into BIPs?
\end{problem}

Finally, we point out that in ~\cite{tewari_nadeau} (see also \cite{KnutsonSanchezSherman}) the authors study decompositions of $\Pi_n$ into cubes, which turn out to be BIPs. Notice that the finest subdivisions of $\Pi_4$ into LPFMs are such that not all of its maximal faces are cubes. However, we would like to understand which LPFMs are cubes.

\bibliographystyle{amsplain}
\bibliography{main.bib}

\end{document}